\documentclass{amsart}
\usepackage{graphicx} 
\usepackage{hyperref}
\usepackage{amsmath, amsthm, amssymb, amscd}
\usepackage{enumerate}
\usepackage{tikz}
\usetikzlibrary{plotmarks}
\usepackage{hyperref}
\usepackage{verbatim}
\usepackage{color}
\usepackage{biblatex} 
\allowdisplaybreaks

\DeclareMathOperator{\diam}{diam}

\theoremstyle{plain}
\newtheorem{theorem}{Theorem}

\newtheorem{corollary}[theorem]{Corollary}

\newtheorem{claim}[theorem]{Claim}

\newtheorem{lemma}[theorem]{Lemma}

\newtheorem{proposition}[theorem]{Proposition}

\theoremstyle{definition}

\newtheorem{definition}[theorem]{Definition}
\newtheorem{remark}[theorem]{Remark}

\numberwithin{equation}{section}
\numberwithin{theorem}{section}

\newcommand{\sparse}{\text{sparse}}
\newcommand{\side}{\text{side}}
\newcommand{\rad}{\text{rad}}

\newcommand{\dist}{\text{dist}}

\newcommand{\RR}{\mathbb{R}}

\newcommand{\B}{\mathcal{B}}

\title{Quantitative metric density and connectivity for sets of positive measure}

\date{\today}
\author{Guy C. David}
\address{Department of Mathematical Sciences\\ Ball State University, Muncie, IN 47306}
\email{gcdavid@bsu.edu}
\date{\today}
\thanks{G. C. David was partially supported by the National Science Foundation under Grant No. DMS-2054004. The main results of the paper form part of the master's thesis of B. Oliva.}

\author{Brandon Oliva}
\address{Department of Mathematical Sciences\\ Ball State University, Muncie, IN 47306}
\email{brandon.oliva@bsu.edu}


\begin{document}

\begin{abstract}
We show that in doubling, geodesic metric measure spaces, (including, for example, Euclidean space) sets of positive measure have a certain large-scale metric density property. As an application, we prove that a set of positive measure in the unit cube of $\mathbb{R}^d$ can be decomposed into a controlled number of subsets that are ``well-connected'' within the original set, along with a ``garbage set'' of arbitrarily small measure. Our results are \textit{quantitative}, i.e., they provide bounds independent of the particular set under consideration.
\end{abstract}
\maketitle

\section{Introduction}

\subsection{Background: qualitative vs. quantitative results in analysis}
Many important theorems in analysis are ``qualitative'' or ``infinitesimal''. They assert that a certain desirable small-scale limiting behavior occurs, but do not provide any guarantee of a fixed scale at which such a behavior occurs up to a given error.

A basic example is Rademacher's theorem: a Lipschitz function from $\RR^d$ to $\RR$ is differentiable almost everywhere. While this is a wonderful result, neither the theorem nor its typical proofs (see, e.g., \cite{mattila, heinonen}) guarantee anything about how far one might have to zoom in at a generic point to find a scale where this Lipschitz function looks like an affine function with at most, say, 1\% error.

With completely different arguments, however, such a guarantee can be made. Let us say that a function $f$ is ``$\epsilon$-close to affine'' on a ball $B$ of radius $r$ if there is an affine function $A$ such that $\sup |f-A|\leq \epsilon r$ on $B$. A result of Dorronsoro \cite{dorronsoro} implies the following quantitative statement.\textit{ For each $d\in\mathbb{N}$ and $\epsilon>0$, there is an $r_0>0$ such that every $1$-Lipschitz $f\colon [0,1]^d\rightarrow\RR$ has a ball of radius at least $r_0$ on which it is $\epsilon$-close to being affine.} (See also expositions by Semmes \cite[B.29]{gromov} and Young \cite[Theorem 2.4]{young}.) 
In fact, more is true: in some sense ``most'' balls in $[0,1]^d$ have the property that $f$ is $\epsilon$-close to affine on them. Note that the parameter $r_0$ is a guarantee that can be made \textit{independent of the particular $1$-Lipschitz function } $f$.

There is by now a whole literature in quantitative analysis that is too large to survey here, including both new phenomena and quantitative analogs of classical results. In addition to the paper of Dorronsoro referenced above, some of the foundational texts are by Jones \cite{jonesTST} and David--Semmes \cite{davidsemmes}.

\subsection{The Lebesgue density theorem}
Another ``qualitative'' result in analysis, in the sense described above, is the Lebesgue density theorem. Using $\lambda$ to denote Lebesgue measure in $\mathbb{R}^n$, a version of the theorem is as follows:
\begin{theorem}[Lebesgue Density Theorem]
Let $E\subseteq \RR^d$ be measurable. Then at almost every point $x\in E$,
$$\lim_{r\rightarrow 0} \frac{\lambda(B(x,r)\cap E)}{\lambda(B(x,r))} =1.$$
\end{theorem}
In other words, at infinitesimally small scales, $E$ ``fills up'' nearly all the measure of a ball. We note that the same statement holds in arbitrary complete, doubling metric measure spaces; see below for the definitions of these terms and \cite{heinonen} for the result in this generality.

As a na\"ive quantitative analog of the Lebesgue density theorem, paralleling the quantitative Rademacher theorem mentioned above, one might hope for the following: 

\begin{equation}\label{eq:false}
\begin{split}
\text{For each } d\in\mathbb{N} \text{ and } \alpha,\epsilon>0\text{, there is an }r_0>0\text{ such that if } E\subseteq [0,1]^d \text{ has }\lambda(E) \geq \alpha,\\
\text{then there is a ball }B\text{ of radius }r\geq r_0\text{ such that }
\frac{\lambda(B \cap E)}{\lambda(B)} \geq 1-\epsilon.
\end{split}
\end{equation}

This statement turns out to be \textbf{false}, however, as we demonstrate with a straightforward counterexample in Section \ref{sec:counterexample}.

\subsection{Quantitative metric density}
To give a correct quantitative statement inspired by the Lebesgue density theorem, we view density in a metric rather than measure-theoretic sense.

\begin{definition}
We define the \textit{sparsity} of a set $E$ in a closed ball $B$ of radius $r$ by
\begin{equation*}
    \sparse(E,B) = \frac{\sup\{\dist(x,E\cap B): x \in B\}}{r}.
\end{equation*}
If $E\cap B = \emptyset$, we regard the sparsity $\sparse(E,B)$ as undefined.
\end{definition}

Thus, if $\sparse(E,B)=0$, then $\overline{E\cap B} = B$, although of course it may be that the measure of $E\cap B$ is $0$. If $\sparse(E,B)$ is small, we consider $E$ to be ``very dense'' in $B$.

In order to formulate the idea that a set has small sparsity at ``most'' locations and scales, we also require the notion of a ``multi-resolution family'' of balls in a metric space.

\begin{definition}
 Let $(X,d)$ be a metric space. Let $\{N_k:k=0,1,2,\dots\}$ be a family of $2^{-k}$-nets in $X$ satisfying $N_0\subseteq N_1 \subseteq N_2 \subseteq \dots$. (See Definition \ref{def:net}.)
 
 Given a constant $A\geq 1$, a \textit{multiresolution family} is the collection
    \begin{equation*}
        \mathcal{B} = \{B(x,A\cdot 2^{-k}) : x\in N_k, 0\le k < \infty\}
    \end{equation*}  
\end{definition}

Note that if $A\geq 1$, then the collection of all balls in $\mathcal{B}$ of fixed radius $A\cdot 2^{-k}$ forms a cover of $X$.

With these definitions, we can state a ``Quantitative Metric Density Theorem''. The terminology employed in the statement (e.g., ``doubling'', ``geodesic'') will be defined in Section \ref{sec:prelim}.
\begin{theorem}\label{thm:QMDT}
Let $(X,d,\mu)$ be a complete, geodesic, metric measure space equipped with a multiresolution family $\mathcal{B}$. Assume that $\mu$ is a C-doubling measure, $\diam{X}=1$, and $\mu(X)=1$.

Given $\epsilon>0$, there is a $K$ (depending only on $\epsilon$, $C$, and $A$) such that, for any set $E\subseteq X$, if we set
\begin{equation*}\label{B(E,e)}
    \mathcal{B}(E,\epsilon)=\{B\in\mathcal{B}:B\cap E\neq \emptyset, \sparse(E,B)\geq \epsilon\}.
\end{equation*}
then
\begin{equation}\label{eq:QMDTsum}
    \sum_{B \in \mathcal{B}(E,\epsilon)}\mu(B) \leq K.
\end{equation}
\end{theorem}

Theorem \ref{thm:QMDT}  says that a set of positive measure must look very dense in ``most'' balls that intersect it, in a way which is \textit{independent of the particular set under consideration}.

For those unfamiliar with the sort of sum over balls of all scales in \eqref{eq:QMDTsum} (sometimes called a ``Carleson packing'' condition), we include the following small piece of intuition: Suppose we did not include the condition that $\sparse(E,B)\geq\epsilon$ in the sum in \eqref{eq:QMDTsum}. In other words, supposed we simply summed $\mu(B)$ over all balls $B\in\mathcal{B}$ intersecting $E$. In that case, the above sum would diverge (if $\mu(E)>0$) because
\begin{equation*}
    \sum_{B\in \mathcal{B}, B\cap E \neq \emptyset}\mu(B) = \sum_{k=0}^{\infty}\sum_{\substack{B\in \mathcal{B}\\ B\cap E \neq \emptyset\\ \rad(B)=A 2^{-k}}}\mu(B)
    \geq \sum_{k=0}^\infty \mu(E) = \infty.
\end{equation*}
Therefore, Theorem \ref{thm:QMDT} says vaguely that $E$ has sparsity $<\epsilon$ in ``most'' balls in $\mathcal{B}$ that intersect it. 

To relate this back to the discussion above, we note that Theorem \ref{thm:QMDT} will easily imply the following:

\begin{corollary}\label{cor:QMDT}
Let $(X,d,\mu)$ be a complete, geodesic metric measure space with $\mu$ a C-doubling measure. Suppose that $\diam{X}=1$ and $\mu(X)=1$.

Given $\alpha$,$\epsilon>0$, there is an $r_0>0$ (depending only on $\alpha, \epsilon, C$) such that if a Borel set $E\subset X$ has $\mu(E)\ge \alpha$, then there is a ball $B$ of radius $r\in [r_0,1]$ such that 
    \begin{equation}\label{sparse small}
       E\cap B \neq\emptyset \text{ and } \sparse(E,B) < \epsilon.
    \end{equation}
\end{corollary}
In other words, we guarantee a large scale (with size independent of $E$) in which $E$ is very dense.

\begin{remark}

The main ideas behind Theorem \ref{thm:QMDT} and Corollary \ref{cor:QMDT} are not really new, although we believe our presentation of the result has new features. The key intermediate step in the proof (Proposition \ref{prop 4.9}) can essentially be found in \cite[Section IV.1.2]{davidsemmes} and \cite[Lemma 1.4]{schul}, although stated with more restrictive assumptions.

Thus, our main purpose in stating and proving Theorem \ref{thm:QMDT} is the application to the original Theorem \ref{thm:connected} below. As a secondary concern, we take the opportunity to provide a general proof of the result in the setting of abstract metric measure spaces.
    
\end{remark}

\begin{remark}
The assumptions that $\diam X=1$ and $\mu(X)=1$ in Theorem \ref{thm:QMDT} and Corollary \ref{cor:QMDT} are not really restrictive, since one can always normalize to these parameters. See Remark \ref{rmk:diameter} for an example.
\end{remark}

\subsection{Quantitative connectivity}
Our main new result is about sets of positive measure in the unit cube of $\mathbb{R}^d$. We show that such a set can be decomposed, up to a ``garbage set'' of small measure, into pieces that are ``well-connected'' within the larger set, in a way we now make precise.

Given $\eta>0$, an \textit{$\eta$-chain} in a metric space is a finite list of points $(x_0, \dots, x_m)$ such that $d(x_i, x_{i+1})<\eta$ for each $i\in\{0,\dots,m-1\}$. We say that the chain is ``from $x$ to $y$'' if $x=x_0$ and $y=x_m$. The \textit{length} of the $\eta$-chain is the sum
$ \sum_{i=0}^{m-1} d(x_i,x_{i+1}).$

\begin{definition}\label{def:well-connected}
    Let $E\subset [0,1]^{d}$, and let $F\subset E$. We say that $F$ is \textit{$\delta$-well-connected in $E$} if for every $x,y\in F$, there is a $\delta|x-y|$-chain from $x$ to $y$ in $E$ with length at most $(1+\delta)|x-y|$.
\end{definition}
Thus, if $F$ is well-connected in $E$, then any pair of points in $F$ can be joined by a discrete path in $E$ whose \textit{steps are small} (compared to the distance between the endpoints) and whose \textit{length is almost as small as possible}.

We show that all sets of positive measure have ``large'' pieces that are well-connected within the original set, in a quantitative way.

\begin{theorem}\label{thm:connected}
    Given $d\in\mathbb{N}$, $\alpha>0$, and $\delta>0$, there is an $M\in\mathbb{N}$ (depending only on $d,\alpha,\delta$) with the following property: 
    If $E\subset [0,1]^d$, then there are sets $F_{1},F_{2},...,F_{M},Z\subset E$ such that:
    \begin{enumerate}
       \item $E=F_{1}\cup F_{2}\cup...\cup F_{M}\cup Z$.
        \item $\lambda(Z)<\alpha$.
        \item For each $i\in \{1,2,...,M\}$, $F_{i}$ is $\delta$-well-connected in $E$.
     
    \end{enumerate}
\end{theorem}
The ``quantitative'' aspect here is the fact that the number $M$ of well-connected subsets depends only on the given parameters, and \textit{not} on the original set $E$.

\begin{remark}
The set $E$ is not assumed to be measurable in Theorem \ref{thm:connected}. However, the set $Z$ can always be taken to be the intersection of $E$ with a measurable subset of $\mathbb{R}^d$ of measure $<\alpha$. Thus, if $E$ is not measurable, then conclusion (2) of the theorem can be interpreted in the sense of Lebesgue outer measure.
\end{remark}

The main steps in the proof of Theorem \ref{thm:connected} are as follows: We apply Theorem \ref{thm:QMDT} (suitably reinterpreted for dyadic cubes rather than balls) to show that in ``most'' dyadic cubes that touch $E$, the sparsity of $E$ is small. (In fact, it is important to use expansions of dyadic cubes, but we elide this for now.) This will imply that most points of $E$, outside of a small garbage set $Z$, lie in a controlled number of ``bad'' dyadic cubes in which $E$ is not very dense.

The next step is a ``coding'' argument, originally due to Jones \cite{jones_1}, to separate $E\setminus Z$ into a controlled number of sets $F_i$ with the following property: If two points $x,y\in E$ are in the same $F_i$, then there is a cube containing them of size $\approx |x-y|$ in which $E$ is very dense. This coding argument is Lemma \ref{lemma 6.6}; we take this opportunity to give a different proof than the ones presented in \cite{jones_1} and \cite{david}.

Finally, we complete the proof by showing that each set $F_i$ is well-connected in $E$.

\begin{remark}
Theorem \ref{thm:connected} bears some resemblance to the ``Checkerboard Theorem'' \cite[Theorem 1.3]{joneskatzvargas} of Jones--Katz--Vargas, but it seems that neither result implies the other. The result of \cite{joneskatzvargas} shows that a set $B$ of positive measure in the unit cube admits a (quantitatively) large subset $A$ that is ``checkerboard connected'' within the original set. This means that any pair of points in $A$ can be joined by a finite sequence of jumps parallel to the coordinate axes whose endpoints lie in $B$, and whose total length is controlled.

In Theorem \ref{thm:connected}, the jumps are not necessarily parallel to the coordinate axes (otherwise the length bound in Definition \ref{def:well-connected} would be impossible). However, the jumps in Theorem \ref{thm:connected} are always \textit{small} compared to the diameter of the path, whereas in \cite[Theorem 1.3]{joneskatzvargas} this is not required. 

The proofs are also different; for instance, \cite{joneskatzvargas} makes essential use of maximal function bounds, and we do not.
\end{remark}

\begin{remark}
It seems likely that a version of Theorem \ref{thm:connected} is true in the generality of doubling, geodesic metric spaces. One main ingredient, Theorem \ref{thm:QMDT}, is already stated and proven in this setting. 

The main change needed to push the argument for Theorem \ref{thm:connected} to this setting would be a version of dyadic cubes in abstract metric spaces. Such a construction exists \cite{christ, hytonen}, but in the interest of keeping the present paper relatively streamlined we do not attempt to give the most general possible result here.
\end{remark}

\subsection{Outline of the paper}
Basic definitions and notations are explained in Section \ref{sec:prelim}. Section \ref{sec:counterexample} contains a counterexample to the quantitative measure-theoretic density statement \eqref{eq:false}. Section \ref{sec:QMDT} contains the proofs of Theorem \ref{thm:QMDT} and Corollary \ref{cor:QMDT}, and Section \ref{sec:connectivity} contains the proof of Theorem \ref{thm:connected}.

\section{Preliminaries}\label{sec:prelim}
The section contains some additional definitions,  notation, and basic facts used in the paper.

\begin{definition}
    A \textit{metric measure space} is a triple consisting of a space $X$, metric $d$, and (outer) measure $\mu$, denoted $(X,d,\mu)$, with an associated $\sigma$-algebra of measurable sets.
\end{definition}
We use $B(x,r)$ to denote the closed ball of radius $r$ centered at $x$, i.e., $\{y\in X: d(x,y)\leq r\}.$

\begin{definition}
    Let $(X,d)$ be a metric measure space and $C>0$. A non-zero Borel regular measure $\mu$ on $X$ is called \textit{doubling} (or \textit{$C$-doubling} to emphasize the constant) if it assigns finite measure to every ball and for each ball $B(x,r) \subset X$ we have
    \begin{equation*}
        \mu(B(x,2r)) \le C\cdot\mu(B(x,r)).
    \end{equation*}
\end{definition}

Of course, Lebesgue measure $\lambda$ on $\mathbb{R}^n$ is an example of a doubling measure.

\begin{definition}
    Let $(X,d)$ be a metric space. A metric space is said to be a \textit{doubling metric space} if there is an $N \in \mathbb{N}$ such that for every ball $B(x,2r) \subset X$, there are $N$ balls $B(x_k,r)$ such that 
    \begin{equation*}
        B(x,2r) \subseteq \bigcup_{k=1}^{N} B(x_k,r).
    \end{equation*}
\end{definition}
In other words, a metric space is a doubling metric space if every ball can be covered by a controlled amount of balls of half its original radius. If $(X,d,\mu)$ is a metric measure space with $\mu$ a doubling measure, then $(X,d)$ is a doubling metric space; see \cite[p. 82]{heinonen}.

\begin{definition}\label{def:net}
    Let $(X,d)$ be a metric space. Given $s>0$, an \textit{$s$-net} is a set $N \subset X$ that satisfies the following:
    \begin{enumerate}
        \item If $x,y\in N$ and $x\neq y$, then $d(x,y)\ge s$.
        \item If $z\in X$, then there is an $x\in N$ such that $d(x,z)<s$.
    \end{enumerate}
\end{definition}
We say that a set is \textit{$s$-separated} if only the first condition above is satisfied.
\begin{definition}
    Let $(X,d_X)$, $(Y,d_Y)$ be two metric spaces, and let $f:X\rightarrow Y$ be a function. We call $f$ an \text{isometric embedding} if for every $x_1,x_2\in X$
    \begin{equation*}
        d_Y(f(x_1),f(x_2))=d_X(x_1,x_2).
    \end{equation*}
\end{definition}
\begin{definition}
    Let $(X,d)$ be a metric space. We say $X$ is \textit{geodesic} if for every $x,y \in X$, there is an isometric embedding $f:[0,d(x,y)]\rightarrow X$ such that $f(0)=x$ and $f(d(x,y))=y$. We call this $f$ a geodesic.
\end{definition}

Our basic examples of doubling, geodesic spaces are $\mathbb{R}^d$ and $[0,1]^d$ with the Euclidean metric. In these settings, we will also need the standard dyadic cube decomposition:
\begin{definition}
    Let $a_1,a_2,...a_d$ be integers between 0 and $2^{-k}-1,k\in\mathbb{N}\cup\{0\}$. We define a \textit{dyadic cube} in $[0,1]^{d}$ by
    \begin{equation*}
        Q=[a_1\cdot 2^{-k},(a_1+1)\cdot 2^{-k})\times[a_2\cdot 2^{-k},(a_2+1)\cdot 2^{-k})\times...\times [a_d\cdot 2^{-k},(a_d+1)\cdot 2^{-k}).
    \end{equation*}
    One can triple the dyadic cube by
    \begin{equation*}
        3Q=[(a_1-1)\cdot 2^{-k},(a_1+2)\cdot 2^{-k})\times[(a_2-1)\cdot 2^{-k},(a_2+2)\cdot 2^{-k})\times...\times [(a_d-1)\cdot 2^{-k},(a_d+2)\cdot 2^{-k}).
    \end{equation*}
By analogy, we can define $5Q$, $7Q$, etc.

If $d=1$, we refer to ``dyadic intervals'' rather than dyadic cubes.
\end{definition}

The side length of a cube $Q$ is written $\side(Q)$. We denote the set of dyadic cubes in $[0,1]^d$ of side length $2^{-k}$ by $\Delta_k$, with $d$ understood from context, and we set $ \Delta = \bigcup_{k=0}^{\infty}\Delta_{k}.$

A \textit{child} of a dyadic cube $Q$ is a dyadic cube $R\subset Q$ with $\side(R) = \frac12 \side(Q)$.

If $Q$ is a dyadic cube in $\mathbb{R}^d$, then $3Q$ is the union of $3^d$ disjoint dyadic cubes of the same side length as $Q$. We refer to these as the \textit{cells} of $3Q$.

A basic fact about dyadic cubes is that, if $Q,R\in \Delta$, then either $Q\cap R =\emptyset$, $Q\subset R$, or $R\subset Q$.

\section{Counterexample to na\"ive quantitative Lebesgue density}\label{sec:counterexample}
Here, we show by a simple $1$-dimensional construction that the quantitative measure-theoretic density statement \eqref{eq:false} cannot hold.

 Let $d=1$, $\alpha =\frac{1}{2}$, and $\epsilon=\frac{1}{10}$. Let $r_0 \in (0,1)$ and $k\in\mathbb{N}$ be the smallest integer such that $2^{-k}\leq r_0$.\\

Let $E_k\subseteq [0,1]$ be defined as:
\begin{equation*}
    E_k=\left(0,\frac{1}{2^{k+2}}\right)\cup \left(\frac{2}{2^{k+2}},\frac{3}{2^{k+2}}\right)\cup...\cup\left(\frac{2^{k+2}-2}{2^{k+2}},\frac{2^{k+2}-1}{2^{k+2}}\right).
\end{equation*}
In other words, $E_k$ is the union of $2^{k+1}$ intervals each of measure $2^{-(k+2)}$, each interval separated by a length of $2^{-(k+2)}$, and the union is contained in $[0,1]$. (See Figure \ref{fig:counter} below.) Of course, $\lambda(E_k)=\frac12 = \alpha$ for all $k$.

\begin{center}
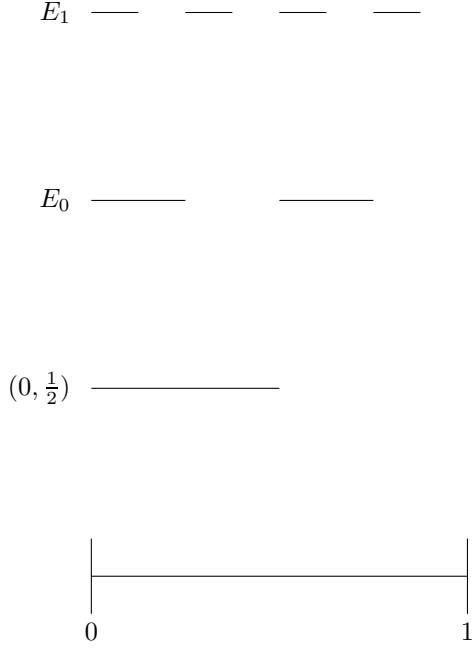
\begin{figure}
\begin{tikzpicture}[scale =5]
    \draw (0,0) -- (1,0); 
    \foreach \x in {0,...,1} {
        \draw (\x,0.1) -- (\x,-0.1) node[below] {\x}; 
    }

    \draw (0,0.5) -- node[left=40pt]{$(0,\frac{1}{2})$} (0.5,0.5);
        
    \draw (0,1) -- node[left = 22pt]{$E_0$} (0.25,1);
    \draw (0.5,1) -- (0.75,1);

    \draw (0,1.5) -- node[left = 13pt]{$E_1$} (0.125,1.5);
    \draw (0.25,1.5) -- (0.375,1.5);
    \draw (0.5,1.5) -- (0.625,1.5);
    \draw (0.75,1.5) -- (0.875,1.5);
    
\end{tikzpicture}

\caption{The sets $E_k$ for $k=0,1$.}\label{fig:counter}
\end{figure}
\end{center}

Note that we have,
\begin{equation*}
    \lambda(E_k)=2^{k+1}\cdot2^{-(k+2)}=\frac{1}{2}.
\end{equation*}

Let $I$ be a dyadic interval in $[0,1]$ such that $2^{-(k+1)}\le \lambda(I) \le 1$. Then, observe that
\begin{equation}\label{Ek I measure}
    \lambda(E_k \cap I) = \frac{1}{2}\lambda(I).
\end{equation}

Take any $x \in [0,1]$ and $r\in [r_0,1]$. We will show that 
\begin{equation*}
   \frac{\lambda(E_k\cap [x-r,x+r])}{\lambda([x-r,x+r])}= \frac{\lambda(E_k\cap B(x,r))}{\lambda(B(x,r))} < 1-\epsilon = 0.9,
\end{equation*}
thus disproving \eqref{eq:false}.

\begin{claim}
    There is a dyadic interval $I\subseteq [x-r,x+r]\cap [0,1]$ such that $\frac{r}{2} \le \lambda(I) \le 2r$
\end{claim} 
\begin{proof}
Let $j \in \mathbb{N}$ be the first integer where $2^{-j}\le r$. Let $I$ be a dyadic interval in $[0,1]$ of length $2^{-j}$ containing $x$. Note that $I\subset [x-r,x+r]$, so clearly $\lambda(I)\le 2r$.

If $2^{-j} < \frac{r}{2}$, then we have
\begin{equation*}
    2^{-(j-1)}=2^{-j+1}< r.
\end{equation*}
This contradicts our choice of $j$. Therefore,
\begin{equation*}
    \frac{r}{2}\leq\lambda(I)=2^{-j}<r\leq 2r.
\end{equation*}
\end{proof}

So since $\frac{r}{2} \le \lambda(I) \le 2r$, observe that
\begin{equation*}
    2^{-(k+1)}\le \frac{r_0}{2}\le \frac{r}{2}\le \lambda(I)\le 2r. 
\end{equation*} 
Therefore by \eqref{Ek I measure} above,
\begin{equation*}
    \lambda(E_k\cap I)=\frac{1}{2}\lambda(I).
\end{equation*}
So,
\begin{equation*}
    \lambda(E_k\cap [x-r,x+r])\leq 2r - \frac{1}{2}\lambda(I)\leq 2r-\frac{r}{4}=\frac{7}{4}r.
\end{equation*}
Hence,
\begin{equation*}
    \frac{\lambda(E_k\cap [x-r,x+r])}{\lambda([x-r,x+r])}\le \frac{7}{4}r\cdot\frac{1}{2r} = \frac{7}{8} < 1-\frac{1}{10}=\frac{9}{10}.
\end{equation*}

Thus, statement \eqref{eq:false} fails in general.

\section{Quantitative metric density}\label{sec:QMDT}
Our goal in this section is to prove Theorems \ref{thm:QMDT} and Corollary \ref{cor:QMDT}.

The following proposition guarantees the existence of nested families of nets, and therefore of multiresolution families, in any metric space. It is well-known, and we omit the standard proof using Zorn's lemma.

\begin{proposition}
Let $(X,d)$ be a metric space. Then there is a sequence of $2^{-k}$-nets $N_k$ in $X$ such that $N_0 \subseteq N_1 \subseteq N_2 \subseteq \dots$  
\end{proposition}

We also need the following basic consequence of the doubling condition.
\begin{lemma}\label{doubling property}
    Let $(X,d)$ be a doubling metric space, and let $N$ be its doubling constant. Let $\mathcal{B}$ be a multiresolution family with constant $A>0$. Then there is a constant $M>0$, depending only on $N$ and $A$, such that for every $x\in X$ and $k\in \{0,1,2...\}$, at most $M$ balls from $\mathcal{B}$ of radius $A2^{-k}$ can contain $x$.
\end{lemma}

\begin{proof}
     Fix $x\in X$ and let $k\in \{0,1,2,...\}$. Let $\{B_1,B_2,B_3,...\}$ be the balls of radius $A\cdot 2^{-k}$ in $\mathcal{B}$ containing $x$, where $B_i=B(x_i,A\cdot 2^{-k})$ for $i\in \{1,2,...\}$ and some $k\in \{1,2,...\}$. Thus, the points $x_i$ are $2^{-k}$-separated and all contained in $B(x,A2^{-k})$. 

Iterating the doubling property of $(X,d)$ tells us that, given $p\in\mathbb{N}$, the ball $B(x, A2^{-k})$ can be covered by at most $N^{p}$ balls of radius $\frac{1}{2^{p}}\cdot A\cdot2^{-k}$.

Therefore, $B(x,A2^{-k})$ can be covered by at most $M$ balls of radius smaller than $\frac{1}{2}\cdot 2^{-k}$ for some $M>0$ depending only on $N$ and $A$. Each such ball contains at most one element of the set $\{x_1,x_2,...\}$, so there at most $M$ distinct elements in this set.
\end{proof}

In order to prove Theorem \ref{thm:QMDT}, we first introduce an intermediate definition.
\begin{definition}
    Let $(X,d,\mu)$ be a metric measure space with $\mu$ Borel regular. Let $E\subset X$ be a nonempty set. Given a ball $B$ of radius $r$, we define
    \begin{equation}\label{d_E}
        d_E(B) = \frac{1}{\mu(B)}\int_{B}\frac{\dist(x,E)}{r}d\mu(x).
    \end{equation}
\end{definition}
This is simply the average distance of a point in $B$ to the set $E$, normalized by the radius of $B$. A simple calculation shows that if $B\cap E \neq \emptyset$, then $0\leq d_E(B)\leq 2.$

We are now ready to state and prove the following proposition that will help us to prove Theorem \ref{thm:QMDT}. As noted in the introduction, the main idea of this proof can already be found in \cite[Section IV.1.2]{davidsemmes} and \cite[Lemma 1.4]{schul}.

\begin{proposition}\label{prop 4.9}
    Let $(X,d,\mu)$ be a doubling metric measure space equipped with a multiresolution family $\mathcal{B}$ of constant $A>0$. Assume $\mu(X)=1$ and $\diam{X}=1$.
    
    Let $E\subset X$. Then
    \begin{equation}\label{prop 4.3}
        \sum\limits_{\substack{B\in \mathcal{B}\\B\cap E\neq \emptyset}}d_E(B)\cdot\mu(B)\leq C_X,
    \end{equation}
    where $C_X$ is a constant depending only on $A$ and the doubling constant of $X$ (and not on $E$).
\end{proposition}
\begin{proof} Note that we have
\begin{equation*}
    \sum\limits_{\substack{B\in \mathcal{B}\\B\cap E\neq \emptyset}}d_E(B)\cdot\mu(B) = \sum\limits_{\substack{B\in \mathcal{B}\\B\cap E\neq \emptyset}} \int_{B}\frac{\dist(x,E)}{r}d\mu(x) = \int_{X}\left(\sum\limits_{\substack{B\in \mathcal{B}\\B\cap E\neq \emptyset\\x\in B}}\frac{\dist(x,E)}{r}\right)d\mu(x),
\end{equation*}
where the $r$ in the summand denotes the radius of the ball $B\in\mathcal{B}$ being summed over.

We will now bound the integrand. Note that if $x\in \bar{E}$, then $\dist(x,E)=0$. So fix $x\in X$ with $x\notin \bar{E}$, and consider
\begin{equation}\label{sum dist}
    \sum\limits_{\substack{B\in \mathcal{B}\\B\cap E\neq \emptyset\\x\in B}}\frac{\dist(x,E)}{r}.
\end{equation}
Let $N$ be the biggest integer such that $A\cdot 2^{-N}\geq \frac{1}{2}\cdot \dist(x,E)$. For balls of radius $r = A\cdot 2^{-N}$, we have 
\begin{equation*}
    \frac{\dist(x,E)}{r}\leq \frac{2\cdot A\cdot 2^{-N}}{r}= \frac{2\cdot A\cdot 2^{-N}}{A\cdot 2^{-N}}=2.
\end{equation*}
Note that there are no balls in the summation in \eqref{sum dist} with radius $A\cdot 2^{-(N+1)},A\cdot 2^{-(N+2)},...$ because such a ball would have radius $r\leq A\cdot 2^{-(N+1)}<\frac{1}{2}\dist(x,E)$, and such a ball could not touch both $E$ and $x$. 
Recall that, since $N$ is the biggest integer such that $A\cdot 2^{-N}\geq \frac{1}{2}\cdot \dist(x.E)$, we have
\begin{equation*}
    \dist(x,E)\leq A\cdot 2^{-(N-1)}.
\end{equation*}
If $j\in \{0,1,2,...,N\}$ and $B$ is a ball in the sum \eqref{sum dist} of radius $A\cdot 2^{-(N-j)}$, then the summand
\begin{equation*}
    \frac{\dist(x,E)}{r}\leq \frac{A\cdot 2^{-(N-1)}}{r}=\frac{A\cdot 2^{-(N-1)}}{A\cdot 2^{-(N-j)}}=\frac{1}{2^{j-1}}.
\end{equation*}
For each fixed radius $A\cdot 2^{-k}$  ($k\in \{0,1,,\dots,N\}$), Lemma \ref{doubling property} tells us that there are at most $M$ balls contaning $x$ with radius equal to $A\cdot 2^{-k}$, where $M$ depends on $N$ and $A$. Therefore we have,
\begin{equation*}
     \sum\limits_{\substack{B\in \mathcal{B}\\B\cap E\neq \emptyset\\x\in B}}\frac{\dist(x,E)}{r}\leq M\cdot(2+1+\frac{1}{2}+\frac{1}{4}+...)\leq M\cdot(\sum_{i=-1}^{\infty}\frac{1}{2^i}) = M\cdot(2+2)=4M.
\end{equation*}
Let $C_X = 4M$. We now have 
\begin{equation*}
    \int_{X}\left(\sum\limits_{\substack{B\in \mathcal{B}\\B\cap E\neq \emptyset\\x\in B}}\frac{\dist(x,E)}{r}\right)d\mu(x) \leq \int_{X}C_X d\mu(x) = C_X \cdot \mu(X) = C_X.
\end{equation*}
Therefore
\begin{equation*}
     \sum\limits_{\substack{B\in \mathcal{B}\\B\cap E\neq \emptyset}}d_E(B)\cdot\mu(B)\leq C_X .
\end{equation*}
\end{proof}
The following lemma connects $\sparse(E,B)$ to $d_E(B)$. There is a minor technical issue that we must overcome, which is that $\sparse(E,B)$ considers only points in $E\cap B$, while $d_E(B)$ considers points of $E$ that may lie outside of $B$. The geodesic assumption allows us to link the two quantities despite this.
\begin{lemma}\label{lemm 4.2}
    Let $(X,d,\mu)$ be a complete, geodesic metric measure space, and $\mu$ a C-doubling measure. Let $E\subset X$ and $B$ be a ball in $X$.
    
    Given $\epsilon >0$, there is a $\delta>0$, depending only on $\epsilon$ and $C$, such that if $d_E(B)<\delta$, then
    \begin{equation*}
        \sparse(E,B)<\epsilon.
    \end{equation*}
\end{lemma}
\begin{proof}
    We will prove this by contrapositive. Without loss of generality, let $0<\epsilon<1$ be given. Suppose we have a metric measure space $(X,d,\mu)$ satisfying the assumptions on the lemma, a subset $E\subset X$, and a ball $B=B(p,r)$ such that $\sparse(E,B)\geq \epsilon$. We will bound $d_E(B)$ from below by a constant depending only on $\epsilon$ and $C$.

    The assumption on sparsity translates to
    \begin{equation*}
        \sup\{\dist(x,E\cap B): x\in B\}\geq \epsilon r.
    \end{equation*}
    Thus, there is a point $x\in B$ such that $\dist(x,E\cap B)> \epsilon \frac{3r}{4}$. The ball $B\left(x,\frac{3\epsilon}{4}r\right)$ therefore contains no points of $E\cap B$.

\begin{claim}\label{claim:y}
There is a point $y\in B\left(x,\frac{\epsilon}{4}r\right)$ such that
$$ B\left(y,\frac{\epsilon}{4}r\right) \subseteq B \cap B\left(x,\frac{3\epsilon}{4}\right)$$
\end{claim}
\begin{proof}    
If $d(p,x)<\frac{\epsilon}{4}$, set $y=x$. The conclusion then follows immediately from the triangle inequality, recalling our assumption that $\epsilon<1$.

If $d(p,x)\geq \frac{\epsilon}{4}$, we do the following. Let $\gamma : [0,d(p,x)]\rightarrow X$ be a geodesic from $x$ to $p$ such that $\gamma(0)=x$, $\gamma(d(p,x))=p$. Set $y=\gamma(\frac{\epsilon}{4}r)$. Then 
$$ d(y,x) = \frac{\epsilon}{4}r$$
    and
$$ d(y,p)=d(p,x)-\frac{\epsilon}{4}r.$$
It then follows easily from the triangle inequality that
$$ B\left(y,\frac{\epsilon}{4}r\right) \subseteq B \cap B\left(x,\frac{3\epsilon}{4}\right)$$

\end{proof}
    
    \begin{claim}
        $\dist(z,E)\geq \frac{\epsilon}{8}r$ for every $z\in B(y,\frac{\epsilon}{8}r)$.
    \end{claim}
    \begin{proof}
    By the triangle inequality and the previous claim, if $z\in B\left(y,\frac{\epsilon}{8}r\right)$, then we have
    $$ B\left(z, \frac{\epsilon}{8}r\right) \subseteq B\left(y, \frac{\epsilon}{4}r\right) \subseteq B\cap B\left(x,\frac{3\epsilon}{4}r\right).$$
    The final set in this inclusion contains no points of $E$, so $B\left(z, \frac{\epsilon}{8}r\right)$ contains no points of $E$, and this proves the claim.
    \end{proof}

We now have   
    \begin{equation*}
        d_E(B)=\frac{1}{\mu(B)}\int_{B}\frac{\dist(z,E)}{r}d\mu(z)\geq \frac{1}{\mu(B)}\int_{B(y,\frac{\epsilon}{8}r)}\frac{\dist(z,E)}{r}d\mu(z)
    \end{equation*}
    \begin{equation*}
        \geq \frac{1}{\mu(B)}\int_{B(y,\frac{\epsilon}{8}r)}\frac{\epsilon r}{8r}d\mu(z)\geq \frac{\epsilon}{8}\cdot \frac{\mu(B(y,\frac{\epsilon}{8}r))}{\mu(B)}.
    \end{equation*}
    If $N$ is the smallest integer such that $2^{N}\cdot \frac{\epsilon}{8}\geq 2$, then $B\subset B(y,2^{N}\cdot \frac{\epsilon}{8}r)$. Because $\mu$ is $C$-doubling,
    \begin{equation*}
        C^{N}\cdot \mu(B(y,\frac{\epsilon}{8}r))\geq \mu(B(y,2^{N}\cdot \frac{\epsilon}{8}r))\geq \mu(B).
    \end{equation*}
    So we have
    \begin{equation*}
        \frac{\mu(B(y,\frac{\epsilon}{8}r))}{\mu(B)}\geq \frac{1}{C^{N}}
    \end{equation*}
    Let $\delta=\frac{\epsilon}{8}\cdot \frac{1}{C^{N}}$, which depends only on $\epsilon$ and $C$. Then 
    \begin{equation*}
        d_E(B)\geq \frac{\epsilon}{8}\cdot \frac{\mu(B(y,\frac{\epsilon}{8}r))}{\mu(B)}\geq \frac{\epsilon}{8}\cdot \frac{1}{C^{N}}=\delta.
    \end{equation*}
\end{proof}
We can now prove Theorem \ref{thm:QMDT}.
\begin{proof}[Proof of Theorem \ref{thm:QMDT}]
    Given $\epsilon>0$, choose $\delta>0$ as in Lemma \ref{lemm 4.2}. This yields that
    \begin{equation*}
    \sum_{B\in \mathcal{B}(E,\epsilon)}\mu(B) \leq \sum\limits_{\substack{B\in \mathcal{B}\\d_E(B)\geq \delta \\ B\cap E\neq \emptyset}}\mu(B)
\end{equation*}
By Proposition \ref{prop 4.9}, we have
\begin{equation*}
    \sum\limits_{\substack{B\in \mathcal{B}\\d_E(B)\geq \delta \\ B\cap E\neq \emptyset}}\mu(B)\leq \sum_{\substack{B\in \mathcal{B}\\ \\ B\cap E\neq \emptyset}}\frac{1}{\delta}d_E(B)\cdot\mu(B) \leq \frac{1}{\delta}\cdot C_X.
\end{equation*}
\end{proof}

\begin{remark}\label{rmk:diameter}
    Suppose that $(Y,d_Y,\mu)$ is a metric measure space with $\mu$ a $C$-doubling measure, $\mu(Y)=1$, but $\diam{Y}=D>1$. Let $\mathcal{B}$ be a multiresolution family in $Y$ consisting of balls of radii $D\cdot A\cdot 2^{-k}$ for $k\ge 0$, and let $E\subset Y$ be measurable. Consider the metric measure space $(X,d_X,\mu)=(Y,\frac{1}{D}d,\mu)$. In the space $X$, the balls of $\mathcal{B}$ have radii $A\cdot 2^{-k}$ for $k\ge 0$. Note in addition that the sparsity of a set in a ball is unaffected by this rescaling, and therefore the notation $\mathcal{B}(E,\epsilon)$ is unambiguous.
    
    Thus given $\epsilon>0$, there is a $K(\epsilon,C,A)>0$ such that if $E\subset Y$ then 
    \begin{equation*}
        \sum_{B\in \mathcal{B}(E,\epsilon)}\mu(B)\le K(\epsilon,C,A).
    \end{equation*}
Thus, a version of Theorem \ref{thm:QMDT} with $\tilde{A}=DA$ holds for metric measure spaces with diameter $D$ greater than 1.\end{remark}
We now prove Corollary \ref{cor:QMDT}.
\begin{proof}[Proof of Corollary \ref{cor:QMDT}:]
    Fix $A=1$ in the definition of our multiresolution family. By Theorem \ref{thm:QMDT}, if $E\subset X$ then
    \begin{equation*}
        \sum_{B \in B(E,\epsilon)} \mu(B)\leq K
    \end{equation*}
    where $K$ depends only on $\epsilon$ and on the doubling measure constant $C$.
    Let $N$ be the first integer such that some ball $B$ in $\mathcal{B}$ of radius $A\cdot 2^{-N}$ touches $E$ and is not in $\mathcal{B}(E,\epsilon)$, that is, $\sparse(E,B)<\epsilon$. This implies that all balls with radii $A\cdot 2^{0}$, $A\cdot 2^{-1}$,...,$A\cdot 2^{-(N-1)}$ that touch $E$ are contained in $\mathcal{B}(E,\epsilon)$. Therefore, the above sum is bounded below by the sum of all the measures of these balls. For each $n\in \{0,1,...,N-1\}$, the measure of the union of all such balls at a given scale is greater than or equal to the measure of $E$. Therefore
\begin{equation*}
    \mu(E)\cdot N \le \sum_{B\in \mathcal{B}(E,\epsilon)}\mu(B) \le K.
\end{equation*}
Note $\mu(E)\geq \alpha$, so we have
\begin{equation*}
    N\cdot \alpha \leq N\cdot \mu(E) \leq K
\end{equation*}
and so
\begin{equation*}
    N \le \frac{K}{\alpha}.
\end{equation*}
Thus, there is a ball $B$ with radius $A\cdot2^{-N}\geq A\cdot2^{-\frac{K}{\alpha}}$ that intersects $E$ and has $\sparse(E,B)<\epsilon$. So we can take $r=A\cdot 2^{-N}$ and $r_0=A\cdot 2^{-\frac{K}{\alpha}}$.
\end{proof}

At this point, it may be instructive to return to the example set(s) $E_k$ in Section \ref{sec:counterexample}, which showed that statement \eqref{eq:false} failed, i.e., that one cannot expect a large ball where a measurable set is measure-theoretically dense.

Corollary \ref{cor:QMDT}, however, applies perfectly well to these sets. For small values of $k$, $E_k$ contains large full dyadic intervals, and thus of course has large scales in which it is metrically dense (i.e., has small sparsity). For large values of $k$, $E_k$ is metrically very dense in the whole interval $[0,1]$, and so again has a large scale where it has small sparsity.

\section{Quantitative connectivity}\label{sec:connectivity}
In this section, we prove Theorem \ref{thm:connected}. This concerns subsets of $[0,1]^d$, rather than general metric spaces.

If $E\subseteq [0,1]^d$ and $Q$ is a dyadic cube such that $7Q\cap E\neq \emptyset$, we set 
\begin{equation*}
    \sparse(E,7Q)=\frac{1}{\side(7Q)}\sup\{\dist(x,7Q\cap E): x\in 7Q\}.
\end{equation*}

We first slightly adjust Theorem \ref{thm:QMDT} to apply to dyadic cubes.

\begin{corollary}\label{QMDT cube}
    Given $\epsilon>0$, there is a $\Tilde{K}>0$, depending only on $\epsilon$ and $d$, such that if $E\subset [0,1]^{d}$, then
    \begin{equation*}
    \sum\limits_{\substack{Q\in \Delta \\ \sparse(E,7Q)\geq \epsilon \\ 3Q\cap E\neq \emptyset}}\lambda(Q)\leq \Tilde{K} .
    \end{equation*}
\end{corollary}
\begin{proof}
    Fix a multi-resolution family of balls $\mathcal{B}$ in $[0,1]^d$ with scaling factor $A=7\sqrt{d}$.
    
    For each $Q\in \Delta_k$, we may choose a ball $B_Q\in \mathcal{B}$ of radius $A 2^{-k} = 7\sqrt{d} 2^{-k}= \diam(7Q)$ such that $7Q\subseteq B_Q$.  

    Now fix a set $E\subset [0,1]^{d}$.
    \begin{claim}
        If $E\cap 7Q\neq\emptyset$, then $\sparse(E,7Q)\leq \frac{A(d+1)}{7}\sparse(E,B_Q)$.
    \end{claim}
    \begin{proof}
Let $x\in 7Q$ and $s=\sparse(E,B_Q)$. Let $k$ be such that $Q\in\Delta_k$, which forces the radius $r$ of $B_Q$ to be $A2^{-k}$.

To prove the claim, it will suffice to find a point $y\in E\cap 7Q$ with $|x-y|\leq (d+1) sr$, which we now do.

If $\dist(x,\partial(7Q))>sr$, then this is easy: there is a $y\in E\cap B$ such that $d(x,y)\leq sr$, and the assumption on the distance to the boundary forces $y\in E\cap7Q$.

Otherwise, $\dist(x,\partial(7Q))\leq sr$. In that case, we first perturb $x$ along $d$ separate line segments parallel to the coordinate axes, each of length at most $sr$, until we reach a point $x'\in 7Q$ with $\dist(x',7Q)>sr$. By the previous case, there is a $y\in 7Q\cap E$ such that $|x'-y|\leq sr$. It follows that
        \begin{equation*}
            |x-y|\leq sr + dsr = (d+1)sr.
        \end{equation*}

Thus,
        \begin{equation*}
            \sparse(E,7Q)\leq \frac{(d+1)sr}{\side(7Q)} = \frac{(d+1)r\sparse(E,B)}{7\cdot 2^{-k}}=\frac{(d+1)A\cdot 2^{-k}\sparse(E,B)}{7\cdot 2^{-k}}
        \end{equation*}
        \begin{equation*}
            =\frac{A(d+1)}{7}\sparse(E,B).
        \end{equation*}
    \end{proof}

Now, each $B_Q$ contains the corresponding cube $7Q$. In addition, each ball $B\in\mathcal{B}$ can act as $B_Q$ for at most $D$ different cubes $Q$, for some $D$ depending on $d$, by a simple volume argument.

Therefore, using the previous claim,
\begin{align*}
 \sum\limits_{\substack{Q\in \Delta \\ \sparse(E,7Q)\geq \epsilon \\ 3Q\cap E\neq \emptyset}}\lambda(Q) &\leq \sum\limits_{\substack{Q\in \Delta \\ \sparse(E,7Q)\geq \epsilon \\ 3Q\cap E\neq \emptyset}}\lambda(B_{Q})\\
 &\leq D \left(\sum\limits_{\substack{B\in \mathcal{B}\\ \sparse(E,B)\geq \frac{7\epsilon}{A(d+1)} \\ B\cap E\neq \emptyset}}\lambda(B)\right) 
\end{align*}
By Theorem \ref{thm:QMDT} (and Remark \ref{rmk:diameter}), this sum is controlled by a constant depending only on $\epsilon$ and $d$, which completes the proof.
\end{proof}
The next result uses Corollary \ref{QMDT cube} to show, roughly speaking, that most points of $E$ are not contained in too many cubes of large sparsity.

\begin{proposition}\label{prop 6.4}
    Let $\epsilon>0$ be given, and suppose $E\subset[0,1]^{d}$. 
    
    For $N\in \mathbb{N}$, let $Z_{N}$ be the collection of points that are contained in at least $N$ distinct cubes $7Q$ such that $Q\in\Delta$, $3Q\cap E\neq \emptyset$, and $\sparse(E,7Q)\ge \epsilon$.
    
    Then $\lambda(Z_{N})\leq 7^d \frac{\Tilde{K}}{N}$, where $\Tilde{K}$ is the constant from Corollary \ref{QMDT cube}.
\end{proposition}

\begin{proof}
 For $N\in \mathbb{N}$ we have
    \begin{equation*}
        Z_{N}=\{x\in [0,1]^{d}: x\textit{ is in at least } N \textit{ distinct dyadic cubes } 7Q \textit{ such that } 3Q\cap E\neq \emptyset \textit{ and } \sparse(E,7Q)\geq \epsilon\}.
    \end{equation*}
The set $Z_N$ is a union of countably many cubes, and therefore Lebesgue measurable. By Corollary \ref{QMDT cube},
\begin{align*}
\tilde{K} &\geq  \sum\limits_{\substack{Q\in \Delta\\ 3Q\cap E\neq \emptyset\\ \sparse(E,7Q)\geq \epsilon}}\lambda(Q)\\
&= 7^{-d}\sum\limits_{\substack{Q\in \Delta\\ 3Q\cap E\neq \emptyset\\ \sparse(E,7Q)\geq \epsilon}}\lambda(7Q)\\
&= 7^{-d}\int_{[0,1]^{d}}\left(\sum\limits_{\substack{Q\in \Delta \\ x\in 7Q \\ 3Q\cap E \neq \emptyset\\ \sparse(E,7Q)\ge \epsilon}}1 \right)dx \\
&\geq  7^{-d}\int_{Z_N}\left(\sum\limits_{\substack{Q\in \Delta \\ x\in 7Q \\ 3Q\cap E \neq \emptyset\\ \sparse(E,7Q)\ge \epsilon}}1 \right)dx \\
&\geq 7^{-d}\int_{Z_N} N d\lambda \\
&= 7^{-d} N\lambda(Z_N).
\end{align*}
The conclusion then follows.
\end{proof}

As preparation for the ``coding'' argument below, we need the following simple fact about dyadic cubes.
\begin{lemma}\label{lem:7Q}
    Let $Q,R\in \Delta$ such that $\side(R)\geq \side(Q)$ and $3Q\cap 3R\neq \emptyset$. Then $3Q\subset 7R$.
\end{lemma}
\begin{proof}
    Let $x\in 3Q\cap 3R\neq \emptyset$. There is a cell $R'\subset 3R$ containing $x$. Then $3Q\subset 5R'\subset 7R$.
\end{proof}

The following lemma provides useful decompositions of sets for which each point is contained in only a controlled number of cubes from a given family. To our knowledge, the idea appears first (in a different context) in \cite[Lemma 2.2]{jones_1} and is also explained in \cite[Lemma 8.4]{david}.

We take the opportunity here to present the lemma in a fairly general form and also give a proof that is somewhat different than those cited above, and which we find a bit more transparent.
    
\begin{lemma}\label{lemma 6.6}
Let $G\subset [0,1]^{d}$, and $B\subset \Delta$. Assume that $3Q\cap G\neq\emptyset$ for every $Q\in B$, and that each $x\in G$ is contained in at most $N$ different cubes $7Q$ for $Q\in B$.

There is a constant $M>0$, depending only on $d$ and $N$, and sets $F_{1},F_{2},...,F_{M}$ with the following properties:
    \begin{enumerate}[(i)]
        \item $G=F_{1}\cup F_{2}\cup...\cup F_{M}$
        \item If $x,y\in F_{i}$ are distinct and $Q$ is a dyadic cube of minimal side length such that $x,y\in 3Q$, then $Q\notin B$. 
    \end{enumerate}
\end{lemma}
\begin{proof}
The proof is split into multiple claims for reading purposes.
    \begin{claim} 
    There are disjoint collections $\mathcal{A}_{1},...,\mathcal{A}_{N}\subset B$ such that the following are true:
    \begin{enumerate}[(i)]
        \item $B=\mathcal{A}_{1}\cup\mathcal{A}_{2}\cup...\cup\mathcal{A}_{N}$.
        \item If $Q,R\in \mathcal{A}_{i}$ are distinct for some $i\in\{1,2,...,N\}$, then $3Q\cap 3R=\emptyset$.
    \end{enumerate}
    \end{claim}
    \begin{proof}
        Order $B=\{Q_{1},Q_{2},...,\}$ by decreasing size, i.e., so that $\side(Q_i)\geq \side(Q_j)$ if $i\leq j$. We will construct the collections $\mathcal{A}_{1},...,\mathcal{A}_{N}$ inductively. Put $Q_{1}\in \mathcal{A}_{1}$.
        
        Assume now that $Q_{1},...,Q_{k}$ have been put in various collections among the $\mathcal{A}_{1},...,\mathcal{A}_{N}$, that the collections are (so far) disjoint from each other, and that if $Q,R\in \mathcal{A}_i$ then $3Q\cap 3R=\emptyset$. We wish to place $Q_{k+1}$ in some $\mathcal{A}_i$  with the property that $3Q_{k+1}\cap 3R=\emptyset$ for all $R$ currently in $\mathcal{A}_i$.
        
        Suppose there is no $i\in \{1,2,...,N\}$ such that $3Q_{k+1}\cap 3R=\emptyset$ for every $R\in \mathcal{A}_{i} \cap \{Q_1, \dots, Q_k\}$. Then $3Q_{k+1}$ intersects distinct cubes $3R_{1},3R_{2},...,3R_{N}$, where $R_{i}\in \mathcal{A}_{i}$ and $\side(R_{i})\geq \side(Q_{k+1})$. By  Lemma \ref{lem:7Q}, $3Q_{k+1}\subset 7R_{i}$ for every $i\in \{1,2,...,N\}$. Then there is an $x\in 3Q_{k+1}\cap G$ such that $x\in 7Q_{k+1},7R_{1},...,7R_{N}$. This contradicts the assumption that $x$ can be contained in at most $N$ different cubes $7Q$ for $Q\in B$.

        Therefore, we may place cube $Q_{k+1}$ in one of the collections $\mathcal{A}_i$ and maintain the desired properties. It follows by induction that the sets may be constructed as desired.
    \end{proof}
    Now let $\alpha = \{0,1,...,3^{d}\}$. We introduce a correspondence between the elements of $\alpha\setminus\{0\}$ and the $3^d$ different cells of a (hence any) tripled dyadic cube $3Q$. The precise choice of correspondence does not matter at all; to be concrete, we can order the cells by the ``dictionary'' order of their centers, so that cell $1$ is the cell closest to the origin, cell $2$ is its neighbor in the $x_1$ direction, etc.
    
    We will consider ``words $w$ of length $N$ from the alphabet $\alpha$,'' i.e., $w\in \alpha^N$. For such a $w$, $w_i\in \alpha$ denotes its $i$th letter.
    
    For each $x\in G$ we define $\omega(x)\in \alpha^{N}$ in the following way. Consider $i\in \{0,\dots, 3^d\}$. The point $x$ can be in at most one cube $3Q$ for $Q\in\mathcal{A}_i$. Let us call this cube $Q_{x,i}$, if it exists. We set the $i$th letter of $w(x)$ to be

    \[ \omega(x)_{i}= \begin{cases}
        0 & \textit{ if } x\notin \bigcup_{Q\in \mathcal{A}_{i}}3Q \\
        k\in \alpha\setminus\{0\} & \textit{ if } x \textit{ is in the kth cell of } 3Q_{x,i}.
    \end{cases}
    \]
Thus, for each $x$, there are at most $M=(3^{d}+1)^{N}$ possibilities for the word $\omega(x)$. Enumerate these possibilities in any order by $\alpha^{N}=\{\omega_{1},...,\omega_{M}\}$. Let $F_{j}=\{x\in G: \omega(x)=\omega_{j}\}$ for $j\in \{1,2,...,M\}$. It is immediate that $G = \cup_{j=1}^M F_j$, so it remains to verify the second conclusion of the lemma.

    \begin{claim}
        Suppose $x,y\in F_{j}$ and $x\neq y$. Let $Q$ be a dyadic cube of minimal side length such that $x,y\in 3Q$. Then $Q\notin B$.
    \end{claim}
    \begin{proof}
        Suppose $Q\in B$. Then $Q\in \mathcal{A}_{i}$ for some $i\in \{1,2,...,N\}$. Since $x,y\in F_{j}$, $\omega(x)=\omega(y)=\omega_{j}$. So $\omega(x)_{i}=\omega(y)_{i}$. It follows from the definition of $\omega(x)_i$ and $\omega(y)_i$ that $x$ and $y$ both lie in the same cell $Q'$ of $3Q$.
        
        Let $R$ be a child of $Q'$ that contains $x$. Then $x,y\in Q'\subset 3R$. Since $R$ is half the size of $Q$, this contradicts the fact that $Q$ is a minimal dyadic cube such that $x,y\in 3Q$. Therefore $Q\notin B$. 
    \end{proof}
Both conclusions of the lemma have now been verified.
\end{proof}
We will also need the following basic fact.
\begin{lemma}
    Let $Q\in \Delta$ be a minimal dyadic cube such that $x,y\in 3Q$. Then $\side(Q)\leq 2|x-y|$.
\end{lemma}
\begin{proof}
    Let $Q'$ be a cell of $3Q$ such that $x\in Q'$. Let $R\subset Q'$ be a child of $Q'$ such that $x\in R$. Because $Q$ is a minimal cube such that $x,y\in 3Q$, it must be that $y\notin 3R$. So 
    \begin{equation*}
        \frac{1}{2}\side(Q)=\side(R)\leq |x-y|.
    \end{equation*}
    This implies that $\side(Q)\leq 2|x-y|$.
\end{proof}
We are now ready to prove Theorem \ref{thm:connected}. 
\begin{proof}[Proof of Theorem \ref{thm:connected}]
    Let $\alpha,\delta>0$. Without loss of generality we may assume $\delta<\frac{1}{2}$. Let $P\in \mathbb{N}$ be chosen large enough so that $\frac{1}{P}\leq \frac{3}{4}\delta$ and $P>1+\frac{1}{\delta}$.  Let $\epsilon=\frac{\delta}{28P}$ and set $B_{0}=\{Q\in \Delta: 3Q\cap E\neq \emptyset, \sparse(E,7Q)\geq \epsilon \}$. For $N\in\mathbb{N}$ to be chosen momentarily, define $Z$ by
    \begin{equation*}
        Z=\{x\in E: x \text{ is in at least } N>0 \text{ different cubes } 7Q \text{ for } Q\in B_{0}\}.
    \end{equation*}

Then $Z= Z_N \cap E$, where $Z_N$ is the (measurable) set defined in Proposition \ref{prop 6.4}. Using that proposition, we have that
    \begin{equation*}
        \lambda(Z)\leq 7^d\frac{\tilde{K}}{N}
    \end{equation*}
    (with the understanding that this denotes Lebesgue outer measure if $E$ is not measurable).
    
For $N$ large enough, depending only on $\epsilon$ and $d$, we have $\lambda(Z)<\alpha$. Let $G=E\setminus Z$, and $B=\{Q\in B_{0}:3Q\cap G\neq \emptyset\}$. Let $x\in G$. Then $x\notin Z$, so $x$ is in at most $N$ different cubes of $B_{0}$ and hence in at most $N$ different cubes of $B$. By Lemma \ref{lemma 6.6}, we can write $G=F_{1}\cup F_{2}\cup \dots \cup F_{M}$ with the property that if $x,y\in F_{i}$ then a minimal dyadic cube $Q$ with $x,y\in 3Q$ satisfies $Q\notin B$. Here $M=(3^{d}+1)^N$.

At this point, we have written
$$ E = G\cup Z = F_1 \cup \dots \cup F_M \cup Z,$$
where $\lambda(Z)<\alpha$. To complete the proof, it remains to check that each set $F_j$ is $\delta$-well-connected in $E$.

Suppose $x,y\in F_{j}$ are distinct for some $j\in \{1,2,...,M\}$. Let $Q$ be a minimal dyadic cube such that $x,y\in 3Q$. Then $Q\notin B$. Since $x,y\in G\cap 3Q$, we have $G\cap 3Q\neq \emptyset$, from whch it follows that $Q\notin B_0$. This implies that $\sparse(E,7Q)<\epsilon$. Note also that $\side(Q)\leq 2|x-y|$ by the previous lemma.

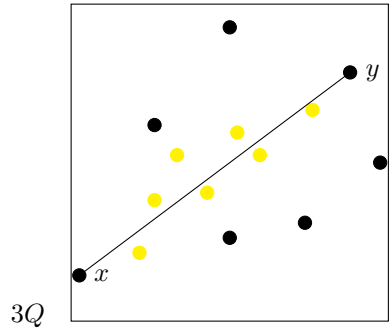
\begin{figure}{$3Q$}
    \begin{tikzpicture}
    \newcommand\Square[1]{+(-#1,-#1) rectangle +(#1,#1)}
    \draw (2,2) \Square{60pt};
    \filldraw (0,0.5) circle[radius=2.5pt];
    \node at (0.3,0.5){$x$};
    
    \filldraw (3.6,3.2) circle[radius=2.5pt];
    \node at (3.9,3.2){$y$};
    
    \draw [black] plot [only marks, mark size=2.5, mark=*] coordinates {(1,2.5) (2,1) (4,2) (3,1.2) (2,3.8)};

    \draw [yellow] plot [only marks, mark size = 2.5, mark = *] coordinates {(0.8,0.8) (1.3,2.1) (1.7, 1.6) (2.4,2.1) (2.1,2.4) (3.1,2.7) (1,1.5)}; 
    
    \draw (0,0.5) -- node[left=40pt][left = 14pt]{ } (3.6,3.2);
\end{tikzpicture}
\caption{An illustration of the chain $\{z_i\}$, highlighted in yellow, built by choosing points of $E$ near the line segment from $x$ to $y$.}\label{fig:pic}
\end{figure}

    Let $\gamma:[0,|x-y|]\rightarrow 3Q$ parametrize the line segment from $x$ to $y$. Let $x_{i}=\gamma(\frac{i}{P}|x-y|)$ for $i\in \{0,1,2,...P\}$. Note that $x_{0}=x$ and $x_{P}=y$. For each $i\in \{1,2,...,P-1\}$, the fact that $\sparse(E,7Q)<\epsilon$ allows us to choose $z_{i}\in E$ such that $|z_{i}-x_{i}|< \epsilon\cdot \side(7Q)$. We also set $z_0=x$ and $z_P=y$, which are in $F_j\subseteq G\subseteq E$ by assumption. Thus, we have a discrete path $\{z_{0}=x,z_{1},z_{2},...,z_{P-1},z_{P}=y\}\subset E$. Figure \ref{fig:pic} illustrates the construction of this path.

We now verify that this path has the desired properties to make $F_j$ $\delta$-well-connected. First, we check, using our choices of $\epsilon$ and $P$, that each ``step size'' is small.

\begin{align}
|z_{i+1}-z_{i}|&\leq |z_{i+1}-x_{i+1}|+|x_{i+1}-x_{i}|+|x_{i}-z_{i}|\nonumber\\
&<2\epsilon \side(7Q) +\frac{1}{P}|x-y|\nonumber\\
&\leq 28\epsilon|x-y| + \frac{1}{P}|x-y|\label{eq:stepbound}\\
&\leq \left(\frac{\delta+1}{P}\right)|x-y|\nonumber\\
&\leq \delta|x-y|\nonumber.
\end{align}

Next, we check that the total length of this path is small. Reusing inequality \eqref{eq:stepbound} from the previous calculation and our definitions of $\epsilon$ and $P$, we obtain that
\begin{align*}
 \sum_{i=0}^{P-1}|z_{i+1}-z_{i}| &\leq  \sum_{i=0}^{P-1}\left(28\epsilon + \frac{1}{P}\right)|x-y|\\
 &= P\left(28\epsilon + \frac{1}{P}\right)|x-y|\\
 &\leq (\delta+1)|x-y|
\end{align*}

Thus, each $F_{j}$ is $\delta$-well-connected in $E$, and this completes the argument.
\end{proof}

\printbibliography

\end{document}